\documentclass[12pt]{amsart}
\usepackage{times}

\usepackage{float}
\usepackage{graphicx}
\usepackage{amsmath,amsthm,amssymb,amscd}
\usepackage{amsfonts}
\usepackage{amsmath,amssymb,amsthm,enumerate,epsfig,graphicx}
\usepackage{ifthen,latexsym,syntonly}
\usepackage{natbib}
\usepackage{amsmath}
\usepackage{amssymb}
\newtheorem{thm}{Theorem}

\newtheorem{lemma}[thm]{Lemma}
\newtheorem{prop}[thm]{Proposition}

\theoremstyle{definition}

\DeclareMathOperator{\var}{Var} 
\newcommand{\zz}[1]{\mathbb{#1}}

\def\to{\rightarrow}

\def\wt{\widetilde}
\def\wh{\widehat}

\def\vep{\varepsilon}

\def\phi{\varphi}

\def\indic{{\mathbf{1}}}

\begin{document}

\title[Critical BRW with Small Drift]{Critical Branching Random Walks
with Small Drift}

\author{Xinghua Zheng}

\address{Department of Information Systems\\
 Business Statistics and Operations Management \\
Hong Kong University of Science and Technology\\
Clear Water Bay, Kowloon, Hong Kong.} \email{xhzheng@ust.hk}

\begin{abstract}
We study critical branching random walks (BRWs) $U^{(n)}$
on~$\mathbb{Z}_{+}$ where for each $n$, the displacement of an
offspring from its parent has drift~$2\beta/\sqrt{n}$ towards the
origin and reflection at the origin. We prove that for
any~$\alpha>1$, conditional on survival to
generation~$[n^{\alpha}]$, the maximal displacement is
asymptotically equivalent to $(\alpha-1)/(4\beta)\sqrt{n}\log n$.
We further show that for a sequence of critical BRWs with such
displacement distributions, if the number of initial particles
grows like~$yn^{\alpha}$ for some $y>0$ and $\alpha>1$, and the
particles are concentrated in~$[0,O(\sqrt{n})],$ then the
measure-valued processes associated with the BRWs, under suitable
scaling converge to a measure-valued process, which, at any
time~$t>0,$ distributes its mass over~$\mathbb{R}_+$ like an
exponential distribution.
\end{abstract}



\date{\today}

\maketitle

\renewcommand{\baselinestretch}{1.0}
\normalsize

\section{Introduction}\label{sec:myIntro} \cite{DKW91} and \cite{Kesten95} studied the
maximal displacement of critical branching random walks (BRWs) on
the real line conditioned to survive for a large number of
generations. When the spatial displacement distribution has drift
$\mu>0$, the results in \cite{DKW91} imply that conditional on the
event that the BRW survives for $n$ generations, the maximal
displacement of a particle from the position of the initial
particle will be of order $O_P(n)$. The main result in
\cite{Kesten95}  asserts that if the spatial displacement
distribution has mean~0 and finite $( 4+\delta)$th moment, then
conditional on the event that the BRW survives for $n$
generations, the maximal displacement will be of
order~$O_P(\sqrt{n})$. The sharp difference between these two
results  gives rise to the following natural question: What
happens if the spatial motions have ``small drift''?

In this paper we supplement these results by showing what happens
for BRWs on the nonnegative integers $\zz{Z}_{+}$ with small
negative drift and reflection at $0$. Assume that $U^{(n)}$ is a
sequence of critical BRWs on the half line
$\zz{Z}_{+}=\{x\in\zz{Z}:x\geq 0\}$, each started by one particle
at the origin, that evolve as follows: (A) At each time
$t=1,2,\dotsc$, particles  produce offspring particles as in a
standard Galton-Watson process with a mean $1$, finite variance
$\sigma^2$ offspring distribution $\mathcal{Q}$. (B) Each
offspring particle then moves from the location of its parent
according to the transition probabilities
$\zz{P}=\zz{P}^{(\beta,n)}$, where $\beta\geq 0$,
\begin{align}\label{eq:Pbeta}
\zz{P}(x,x+1)&=\frac{1}{2} -\frac{\beta}{\sqrt{n}} \quad \text{for} \;\; x\geq 1;\\
\notag
\zz{P} (x,x-1)&=\frac{1}{2} +\frac{\beta}{\sqrt{n}} \quad \text{for} \;\; x\geq 1;\\
\notag \zz{P}(0,1)&=1.
\end{align}
The spatial motion is hence slightly biased towards the origin,
which serves as a reflecting barrier. Such a BRW can be used to
model, for example, a branching process occurring in a V-shaped
valley, where the particles, due to gravity, have a slight
tendency to move towards the bottom. In \cite{Kac47} the
afore-described slightly biased random walk is used to model the
motion of ``heavy Brownian particles'' in a container with its
bottom as a reflecting barrier. \cite{Kac47} also states about the
reflecting barrier that ``the elucidation of its influence on the
Brownian motion is of considerable theoretical interest''. In this
article we will study the influence of the barrier on the BRW.

Denote by $U^{(n)}_{t} (x)$ the number of particles in the $n$th
BRW $U^{(n)}$ at location~$x$ at time $t$, and by $R^{(n)}_{t}$
the location of the rightmost particle at time $t$. Our main
interest is in the conditional distribution of
$R^{(n)}_{[n^{\alpha }]}$ given that the process $U^{(n)}$
survives for $[n^{\alpha}]$ generations. For $\alpha <1$, the
effect of the drift $-2\beta /\sqrt{n}$ will be negligible
compared to diffusion effects over this time interval, and for
$\alpha =1$ it is just large enough to match the diffusion
effects. Thus, we will focus on the case  when $\alpha
>1$.

\begin{thm}\label{thm:range_super}
When $\beta>0$, for each $\alpha>1$ and $\varepsilon >0$, the
range $R^{(n)}_{[n^\alpha]}$ at time $[n^\alpha]$ satisfies
\begin{equation}\label{eq:main}
\lim_{n\to\infty}
   P\left(\left.\left|\frac{R^{(n)}_{[n^\alpha]}}{\sqrt{n}\log    n
   }-\frac{\alpha-1}{4\beta}\right|\geq \vep\;\right|
   G^{(n)}_{[n^\alpha]}\right) =0,
\end{equation}
where for any $k\in\zz{Z}_+,$
\begin{equation}\label{eq:surval}
  G^{(n)}_{k} = \{U^{(n)}\text{ survives to generation }
  k\}.
\end{equation}
\end{thm}

It is natural to consider in connection with the behavior of the
maximal displacement the process-level scaling behavior of the
BRWs.  To this end, consider a series of BRWs $\{X^{(n)}\}$ on the
set $\zz{Z}_{+}$ of nonnegative integers that evolve by the rules
described above, but with arbitrary initial states $X^{(n)}_{0}$.
(In Theorem~\ref{thm:range_super} the initial state consisted of a
single particle located at the origin $0$.)  For integers $x,k\geq
0$, set
\begin{equation}\label{eq:Xnk}
     X_{k}^{(n)}(x)=\#\text{ particles at } x \text{ at time } k.
\end{equation}
For any subset $I\subseteq \zz{R}_{+}$, let
\[
   X_{k}^{(n)}(I) = \sum_{x\in I}X_{k}^{(n)}(x).
\]
Finally, let
\[
   Z_k^{(n)}= X_{k}^{(n)}(\zz{Z}_+) = \sum_{x} X_{k}^{(n)}(x).
\]
Recall that by Kolmogorov's estimate for critical Galton-Watson
processes (see \eqref{eq:survivalToGenM} below), if the $n$th BRW
$X^{(n)}$ is initiated by $O(n^{\alpha})$ particles, then the
total lifetime of the process will be on the order of $O_{P}
(n^{\alpha})$ generations. If $\alpha <1$, then  the effect of the
drift over a time interval $[0,O(n^{\alpha})]$ is too small to be
felt. If $\alpha = 1$ then the drift will be just large enough to
be felt, and so for large $n$ the BRW $X^{(n)}$, suitably
rescaled, will look like a Dawson-Watanabe process on the halfline
$[0,\infty )$ with drift $-2\beta$ and reflection at $0$ (for the
convergence of ordinary BRWs to Dawson-Watanabe processes, see
\cite{watanabe68}, or \cite{etheridge,PerkinsSFnotes}). The case
we will focus on is again when $\alpha>1$, as in this case the
effect of the reflecting barrier at~$0$ dominates the diffusion
effects over the lifetime of the branching process, and the result
is an entirely different scaling behavior:

\begin{thm}\label{thm:measure_limit_super}
When $\beta>0$, assume that for some $\alpha >1$,
\begin{equation}\label{eq:total_mass_conv}
 \frac{Z^{(n)}_{0}}{n^\alpha} \to y>0,\quad\mbox{ as } n
\rightarrow \infty,
\end{equation}
and $\{X^{(n)}_{0}(\sqrt{n}\cdot)/n^\alpha\}_{n\geq 1}$ is tight,
i.e., for any $\vep>0$ there exists $C>0$ such that for all $n$,
\begin{equation}\label{eq:ini_measure_tight}
 \frac{X^{(n)}_{0}([C\sqrt{n},\infty))}{n^\alpha}\leq \vep.
\end{equation}
Then the measure-valued processes $\left(X^{(n)}_{[n^\alpha
t]}(\sqrt{n}\cdot)/n^\alpha:\ t>0\right)$ converge, in the sense
of convergence of finite-dimensional  distributions,  to a process
$(X_t: t>0)$, where $(X_t)_{t\geq 0}$ is such that for all $t\geq
0$ and $0\leq a<b$,
\begin{equation}\label{eq:theWeakLimit}
  X_t((a,b))= Y_t\cdot (\exp(-4\beta a) - \exp(-4\beta b)):=Y_t\cdot \pi((a,b)).
\end{equation}
Here $Y_{t}$ is the \emph{Feller diffusion}:
\begin{equation}\label{eq:feller}
    dY_t=\sigma\sqrt{Y_t}\,dW_t,\quad Y_0=y.
\end{equation}
\end{thm}

Observe that we \emph{do not} require the initial measures
$X^{(n)}_{0}(\sqrt{n}\cdot)/n^\alpha$ to converge; what we only
require are (i) the total mass converges, and (ii) the particles
are not too spread out. In particular, we cannot guarantee that
$X^{(n)}_{0}(\sqrt{n}\cdot)/n^\alpha \Longrightarrow X_0$. Theorem
\ref{thm:measure_limit_super} says that one has finite dimensional
convergence on $(0,\infty).$

The Feller diffusion $(Y_t)$ defined by \eqref{eq:feller} is the
limit of $(Z^{(n)}_{[n^\alpha t]}/n^\alpha )$:
\begin{equation}\label{eq:feller_conv}
   \left( \frac{Z^{(n)}_{[n^\alpha t]}}{n^\alpha}\right) \Rightarrow \left(Y_t\right) \quad\mbox{ on  }
    D([0,\infty);\zz{R}),
\end{equation}
see \cite{Feller39,Feller51}. See Chapter XI of \cite{RevuzYor}
for some basic properties of the Feller diffusion. The limiting
process $X_t$ hence can be described in this way:  its total mass
evolves like the Feller diffusion $Y_{t}$, but the
\emph{distribution} of the mass~$Y_{t}$ at any time $t>0$ is
always the exponential distribution $\pi$. As is proved in
\cite{Kac47}, the exponential distribution $\pi$ is the stationary
distribution of a diffusion process obtained by suitably
normalizing the RWs defined by \eqref{eq:Pbeta} and taking limit
as $n\to\infty$.

The following elementary relation between the expected number of
particles at a site $y$ in generation $m$  for a critical BRW and
the $m$-step transition probability $P(S_m = y)$ of the random
walk will be frequently used: if the critical BRW is started by
one particle at site $x$, and $U_{m} (y)$ stands for the number of
particles at site $y$ in generation $m$, then
\begin{equation}\label{eq:fundamental}
    EU_{m} (y)= P(S_m = y\,|\, S_0=x).
\end{equation}
This is easily proved by induction on $m$, by conditioning on the
first generation and using the fact the the offspring distribution
has mean 1.

The structure of this article is as follows: in Section 2 we prove
some properties of the random walks on the half line, in Section 3
we prove Theorem \ref{thm:range_super}; Theorem
\ref{thm:measure_limit_super} is proved in Section 4.

\medskip \noindent \textbf{Notation.}
We follow the custom of writing $f\sim g$ to mean that the ratio
$f/g$ converges to~1. For any $a,b\in\zz{R}$, $a\wedge
b:=\min(a,b)$ and $a\vee b:=\max(a,b)$. Throughout the paper, $c,
C$ \hbox{etc.} denote generic constants whose values may change
from line to line. For any $x\geq 0$, $[x]$ denotes its integer
part, i.e., the greatest integer no greater than $x$. The notation
$Y_{n}=o_{P} (f (n))$ means that $Y_{n}/f (n) \rightarrow 0$ in
probability; and $Y_{n}=O_{P} (f (n))$ means that the sequence
$|Y_{n}|/f (n)$ is tight.

\section{Random Walks} Throughout this article we use the notation
$\{S_m\}_{m\geq 0}=\{S_m^{(\beta,n)}\}$ to denote a random walk
with transition probabilities $\zz{P}=\zz{P}^{(\beta,n)}$ defined
by equation \eqref{eq:Pbeta}; use $\{\wt{S}_m\}_{m\geq 0}$ to
denote the simple random walk on $\zz{Z}_{+}$ with reflection at
$0$; and use $\{\wh{S}_m\}_{m\geq 0}$ to denote the simple random
walk on $\zz{Z}$. Furthermore, for any such random walks, e.g.,
$\{S_m\}$, for any $x,y\in\zz{Z}_+$ and $m\in\zz{N}$, $P^x(S_m=y)
=P(S_m=y\,|\,S_0=x) $ is the probability that $S_m$ started at $x$
finds its way to site $y$ in $m$ steps.

The following lemma says that the random walk $S_m$ which has
drift towards the origin is stochastically dominated by the
reflected simple random walk $\wt{S}_m$.
\begin{lemma}\label{lemma:coupling}For any $\beta>0$,
$n\in\zz{N}$ and  $x\in\zz{Z}_{+}$, we can build random walks
$\{S_m\}_{m\geq 0} \sim \zz{P}^{(\beta,n)}$ and
$\{\wt{S}_m\}_{m\geq 0}\sim \wt{\zz{P}}$ on a common probability
space so that
$$S_0=\wt{S}_0=x, \text{and } S_m\leq \wt{S}_m,\ \mbox{for all}\ m.$$
\end{lemma}
\begin{proof} It suffices to prove the result for  the case $x>0$; the case
$x=0$ then follows since $ S_1=\wt{S}_1=1.$

Let $S_0=\wt{S}_0=x$. At time 1 sample a $U_1\sim\text{ Unif
}(0,1)$. If $U_1\leq 1/2 +\beta/\sqrt{n}$, then let $S_1=x-1$,
otherwise let $S_1=x+1$. In the meanwhile, if $U_1\leq 1/2 $, then
let $\wt{S}_1=x-1$, otherwise let $\wt{S}_1=x+1$. Clearly
$\{S_0,S_1\}$ and $\{\wt{S}_0,\wt{S}_1\}$ follow their laws
respectively and $S_1\leq \wt{S}_1$. Now suppose that we have
built $\{S_m\}$ and $\{\wt{S}_m\}$ up to time $m$, and we have
$S_m\leq \wt{S}_m$. If $S_m<\wt{S}_m$, we must have $S_m\leq
\wt{S}_m -2$ since at each step the difference between the jumps
is either 0 or 2; now because at each step the random walks can at
most jump 1, at time $m+1$, we must still have $S_{m+1}\leq
\wt{S}_{m+1}$. In the other case when $S_m=\wt{S}_m$, if $S_m>0$
then we can build $S_{m+1}\leq \wt{S}_{m+1}$ just as at time 0;
otherwise $S_m=0$, then necessarily $S_{m+1}=\wt{S}_{m+1}=1$.
Thus, we have proved that we can build $\{S_m\}$ and
$\{\wt{S}_m\}$ up to time $m+1$. By induction, the conclusion
holds.
\end{proof}

\begin{lemma}\label{lemma:sy_srw} For any
$k\in\zz{N}$, any $x\geq k$, and any $m\geq 0,$
\begin{equation}\label{eqn:reflection}
   P^x(\wt{S}_m\geq x+k)
   \leq  P^0(\max_{i\leq m}|\wh{S}_i|\geq k).
\end{equation}
Moreover, there exist $ C>0$ and $ b>0$ such that
\begin{equation}\label{eqn:exp_decay}
   P^0\left(\max_{i\leq m}|\wh{S}_i|\geq k\right)\leq C\exp\left(-\frac{bk^2}{m}\right),\ \mbox{for all}\ m.
\end{equation}
\end{lemma}
\begin{proof}
Inequality (\ref{eqn:reflection}) holds because in order that
$\wt{S}_m\geq x+k$, either the random walk $\{\wt{S}_i\}_{i\leq
m}$ has never visited 0, in which case it just evolves like a
simple random walk whose maximal deviation from $x$ is no less
than $\wt{S}_m-\wt{S}_0\geq k$, or the random walk
$\{\wt{S}_i\}_{i\leq m}$ has visited 0 in which case it evolves
like a simple random walk before hitting 0, and the maximal
deviation from $x$ before time $m$ is no less than $x \geq k$.

Now let us prove (\ref{eqn:exp_decay}). First recall the fact that
for the simple random walk $\{\wh{S}_m\,|\,\wh{S}_0=0\}$, there
exists $b>0$ such that \begin{equation}\label{eqn:expbd_srw}\sup_m
E\exp\left(b\frac{|\wh{S}_m|^2}{m}\right):=C <\infty,\end{equation}
see, e.g., Exercise 2.6 in \cite{lawler07}. Now by the submartingale
maximal inequality,  we get
\[
   P^0\left(\max_{i\leq m}|\wh{S}_i|\geq k\right)
   = P^0\left(\max_{i\leq m}\;\exp(\theta|\wh{S}_i|^2)   \geq \exp(\theta k^2)\right)
   \leq \frac{E\exp(\theta |\wh{S}_m|^2 )}{\exp(\theta k^2)}.
\]
Inequality (\ref{eqn:exp_decay}) follows by taking $\theta$ to be
$b/m$ and using (\ref{eqn:expbd_srw}).
\end{proof}

Next lemma indicates that if two random walks $S_m^1$ and $S_m^2$
have the same drift $2\beta/\sqrt{n}$ towards the origin, and are
such that $S_0^2-S_0^1$ is a positive even number, then $S_m^1$ is
stochastically dominated by $S_m^2$.
\begin{lemma}\label{lemma:comparison} For any fixed $\beta>0$,
$n\in\zz{N},$  $0\leq i_1\neq i_2$, and a random walk
$\{S_m^1\}_{m\geq 0} \sim \zz{P}^{(\beta,n)}$ with $S^1_0=2i_1$,
we can build a coupling random walk $\{S_m^2\}_{m\geq 0}\sim
\zz{P}^{(\beta,n)}$ with $S^2_0=2i_2$ on a possibly extended
probability space such that
\begin{equation}\label{eq:RW_dominance}
  \left\{\aligned
   &S^1_m \leq S^2_m,\ \mbox{for all}\ m, & \mbox{ if } i_1 <
   i_2\\
   &S^1_m \geq S^2_m,\ \mbox{for all}\ m, & \mbox{ if } i_1 >
   i_2\\
  \endaligned\right.
\end{equation}
Similar conclusion holds if we change the initial positions of
$\{S_m^1\}$ and $\{S_m^2\}$ to  $S^1_0=2i_1+1,  S^2_0=2i_2+1.$
\end{lemma}
\begin{proof}
We shall only prove for the case where $S^1_0=2i_1,  S^2_0=2i_2$
and $i_1<i_2$. We will build $\{S_m^2\}$ step by step: if
$S_m^1>0$, then $S_{m+1}^2$ moves in the same direction away from
$S_{m}^2$ as $S_{m+1}^1$ does, i.e.,
\[
   S_{m+1}^2=S_{m}^2 + (S_{m+1}^1 - S_{m}^1);
\]
otherwise if $S_m^1=0$, then choose $S_{m+1}^2$ according to
distribution \eqref{eq:Pbeta}.  Since $S_0^2-S_0^1=2(i_2-i_1)$ is
even and at each step the difference between the jumps is either 0
or 2, the two random walks cannot cross each other and will either
never meet, or merge after they meet. The dominance
\eqref{eq:RW_dominance} follows.
\end{proof}


We now look more closely at the random walks $\{S_m\}\sim
\zz{P}=\zz{P}^{(\beta,n)}$. Based on the results in \cite{Kac47}
we show the following.
\begin{prop}\label{prop:RW_conv_rate} For any fixed $\beta>0$, $a\geq 0$,
and any nonnegative integer sequences $\{s_n\}$, $\{m_n\}$  with
$s_n=O(\sqrt{n})$ and $\lim_n m_n/(n(\log n)^2)>0$, the random
walks $\{S_m^{(n)}\,|\,S_0^{(n)}=s_n\,\}\sim \zz{P}^{(\beta,n)}$
satisfy
\begin{equation}\label{eq:conv_normal_range}
  \lim_{n\to\infty} P\left(S_{m_n}^{(n)}\geq a\sqrt{n}\,|\,S_0^{(n)}=s_n\right) =
  \exp(-4\beta  a),
\end{equation}
and
\begin{equation}\label{eq:conv_outer}
   \lim_{n\to\infty}\frac{P\left(S_{m_n}^{(n)}\geq a\sqrt{n}\log n\,|\,S_0^{(n)}=s_n\right)}{n^{-4\beta a}} = 1.
\end{equation}
\end{prop}
\begin{proof}
When $a=0$, \eqref{eq:conv_normal_range} and \eqref{eq:conv_outer}
clearly hold. So below we assume that $a>0$.

Let
\[
   q=q^{(n)} =
   \frac{1}{2}-\frac{\beta}{\sqrt{n}},\quad\mbox{and}\quad
   p=p^{(n)} =
   \frac{1}{2}+\frac{\beta}{\sqrt{n}}.
\]
 By (41) in \cite{Kac47}, for any $k>0$,
\begin{equation}\label{eq:biased_RW_TP}
\aligned
  &P(S_{m_n}^{(n)} = k\,|\,S_0^{(n)}=s_n)\\
  &=\frac{p-q}{2pq}\left(\frac{q}{p}\right)^k\left(1+(-1)^{s_n+k+m_n}\right)\\
  & + \frac{2}{\pi}\left(\frac{p}{q}\right)^{s_n/2}\left(\frac{q}{p}\right)^{k/2}
  \left(2\sqrt{pq}\right)^{m_n}\int_0^\pi \cos^{m_n}\theta
  \frac{\tan^2\theta}{(p-q)^2 +
  \tan^2\theta}f_{s_n}(\theta)f_k(\theta)\,d\theta\\
  &:=p^*_{m_n}(k)+R_{m_n}(k),
\endaligned
\end{equation}
where for any $i\geq 1$,
\[
   f_i(\theta) = \cos i\theta - 2\frac{\beta}{\sqrt{n}}\frac{\sin
   i\theta}{\sin \theta},\quad\theta\in[0,\pi].
\]
We first estimate the main term $p^*_{m_n}(k)$. Depending on
whether $s_n + m_n$ is even or odd, $S_{m_n}^{(n)}$ only takes
even or odd values. We shall only deal with the case when $s_n +
m_n$ is even. In this case,
\[
   \sum_{k\geq a\sqrt{n}}p^*_{m_n}(k)
   =2\frac{p-q}{2pq} \sum_{k\geq a\sqrt{n},\ k\mbox{ even}}
   \left(\frac{q}{p}\right)^k.
\]
Using the sum formula for geometric series and noting that
\[
  \frac{q}{p}=\frac{\frac{1}{2} - \frac{\beta}{\sqrt{n}}}{\frac{1}{2} +
  \frac{\beta}{\sqrt{n}}} \sim 1-\frac{4\beta}{\sqrt{n}},
\]
one can easily  show that
\begin{equation}\label{eq:conv_normal_range_mt}
  \lim_{n\to\infty}\sum_{k\geq a\sqrt{n}} p^*_{m_n}(k)=
  \exp(-4\beta  a).
\end{equation}
Similarly,
\begin{equation}\label{eq:conv_outer_mt}
  \lim_{n\to\infty}\frac{\sum_{k\geq a\sqrt{n}\log n}  p^*_{m_n}(k)}{n^{-4\beta a}} = 1.
\end{equation}
It remains to show that the remainder terms $R_{m_n}(k)$ decay
rapidly as $n\to\infty$. In fact, by the simple bound
\[
  |\sin i\theta |\leq i\sin \theta,\quad\mbox{for all }\theta \in
  [0,\pi],
\]
we get
\[
  |f_i(\theta)|\leq 1+\frac{2\beta }{\sqrt{n}} i.
\]
Hence, since $s_n=O(\sqrt{n})$,
\[
\aligned
   |R_{m_n}(k)|
   &\leq C\left(2\sqrt{pq}\right)^{m_n}\cdot \left(\frac{q}{p}\right)^{k/2} (1+2k\beta)\\
   &\leq C\exp(-2\beta^2 m_n/n)\cdot
   \exp(-k\beta/\sqrt{n})(1+2k\beta).
\endaligned
\]
As
\[
   \sum_{k=1}^\infty \exp(-k\beta/\sqrt{n})(1+2k\beta) =
   O(\sqrt{n}),
\]
and $\lim_n m_n/(n(\log n)^2)>0$, \eqref{eq:conv_normal_range} and
\eqref{eq:conv_outer} follow from \eqref{eq:conv_normal_range_mt}
and \eqref{eq:conv_outer_mt}.
\end{proof}

\section{Proof of Theorem \ref{thm:range_super}}
We first recall some well known facts about critical Galton-Watson
processes. Let $\sigma^2<~\infty$ be the variance of  the
offspring distribution $\mathcal{Q}$. Then, if $Z_m$ is the number
of particles at time~$m$ with $Z_0=1$,  and $G_{m} = \{Z_m
> 0\}$ is the event that the Galton-Watson
process survives to generation~$m$, then
\begin{align}\label{eq:var_Z_m}
   \var(Z_m) &= m\sigma^2,\\
  \label{eq:survivalToGenM}
    \rho_m:&=  P (G_{m}) \sim \frac{2}{m\sigma^{2}},\\
   \label{eq:exp_surv}
   E(Z_m\,|\,G_m)&=\frac{1}{\rho_m}\sim \frac{\sigma^2m}{2},\;\text{and}\\
   \label{eqn:lim_exp}
   \mathcal{L}\left(\left.\frac{Z_m}{m}\right| G_m\right)
   &\Longrightarrow {\rm Exp}(\sigma^2/2),
   \quad{\rm as }\quad m\to\infty,
\end{align}
see, e.g., sections I.2 and  I.9 of \cite{athreya72}. Relation
\eqref{eq:survivalToGenM} is known as Kolmogorov's estimate;
\eqref{eqn:lim_exp} is Yaglom's theorem.

We will decompose the proof of Theorem \ref{thm:range_super} into
two steps. In Proposition \ref{prop:lower_bd} we show that for any
$\vep>0$, $(\alpha-1-\vep)\sqrt{n}\log n /(4\beta)$ is an
asymptotic lower bound for $R^{(n)}_{[n^\alpha]}$. Proposition
\ref{prop:upper_bd} says that $(\alpha-1+\vep)\sqrt{n}\log n
/(4\beta)$ is an asymptotic upper bound.
Theorem~\ref{thm:range_super} follows by combining these two
propositions.

To prove Theorem \ref{thm:range_super}, we will follow the
strategy used to prove Theorems 4 and 5 in \cite{lz07}, namely,
changing the conditional event $G_k$ to some event defined with
respect to a generation $m(k)\leq k$.

The following two lemmas are Lemmas 17 and 18 in \cite{lz07}.
\begin{lemma}\label{lemma:conditioning} Suppose that on some probability
space $(\Omega, \mathcal{F}, P)$ there are two events $E_1, E_2$
with $P(E_1)P(E_2)>0$ such that
\begin{equation}\label{eqn:sym_diff}
    \frac{P(E_1\Delta E_2)}{P(E_{1})}\leq \varepsilon,
\end{equation}
where $E_1\Delta E_2$ is the symmetric difference of $E_1$ and
$E_2$. Then
\begin{equation}\label{eqn:tv_diff}
   ||P(\cdot|E_1) - P(\cdot|E_2)||_{TV}
   \leq 2\varepsilon,
\end{equation}
where $P(\cdot|E_i)$ denotes the conditional probability measure
given the event $E_{i}$, and $||\cdot||_{TV}$ denotes the total
variation distance.
\end{lemma}


\begin{lemma}\label{lemma:timeJiggle}
Let $m(k)\leq k$ be integers and $\varepsilon_{k}>0$  be real
numbers such that $m(k)/k \rightarrow 1$ and $\varepsilon_{k}
\rightarrow 0$ as $k \rightarrow \infty$. Then
\begin{equation}\label{eq:noAsyDiff}
    \lim_{k \rightarrow \infty} \frac{P (G_{k}\Delta
    H_{k})}{P(G_{k})}=0,
\end{equation}
where
\[
   G(k) = \{Z_k > 0\}\quad\mbox{ and
   }\quad
   H(k) = \{Z_{m(k)} \geq
   k\vep_k\}.
\]
\end{lemma}

By Lemmas \ref{lemma:conditioning} and \ref{lemma:timeJiggle}, we
can change the conditioning event $G_k=\{Z_k>0\}$ to
$H_k=\{Z_{m(k)}\geq k\vep_k\}$, and it suffices to prove the
convergence in Theorem~\ref{thm:range_super} when the conditioning
event is $H_{k}$ rather than $G_{k}$. The advantage of this is
that, conditional on the state of the BRW at time $m (k)$, the
next $k-m (k)$ generations are gotten by running
\emph{independent} BRWs for time $k-m (k)$ starting from the
locations of the particles in generation $m (k)$.

We now show that $(\alpha-1-\vep)\sqrt{n}\log n /(4\beta)$ is an
asymptotic lower bound for $R^{(n)}_{[n^\alpha]}$.
\begin{prop}\label{prop:lower_bd}For any $\vep>0$,
$$
   \lim_{n\to\infty} P\left(\left.R^{(n)}_{[n^\alpha]}\geq \frac{\alpha-1-\vep}{4\beta}\cdot\sqrt{n}\log n
   \,\right|\,G^{(n)}_{[n^\alpha]}\right)=1.
$$
\end{prop}
\begin{proof}
By Lemmas \ref{lemma:conditioning} and \ref{lemma:timeJiggle}, we
can change the conditioning event from $G^{(n)}_{[n^\alpha]}$ to
$\{Z^{(n)}_{[n^\alpha]-nL(n)}>[n^\alpha /L(n)]\}$ for
$L(n):=[(\log n)^2]$, where for any $k\geq 0$, $Z^{(n)}_k$ is the
number of particles at generation $k$ for the $n$th BRW $U^{(n)}$.
Conditioning on $\{Z^{(n)}_{[n^\alpha]-nL(n)}>[n^\alpha /L(n)]\}$,
there will be at least $X\sim \text{Bin}([n^\alpha /L(n)],
\rho_{nL(n)})$ number of particles at time $[n^\alpha]-nL(n)$
whose families will survive to time $[n^\alpha]$. For any such
particle, among its descendants at time $[n^\alpha]$ we uniformly
pick one, then the trajectory of the chosen particle from time
$[n^\alpha]-nL(n)$ to $[n^\alpha]$ will be a random walk following
the law $\zz{P}^{\beta,n}$, starting at the location of its
ancestor at time $[n^\alpha]-nL(n)$. In this way we get at least
$\text{Bin}([n^\alpha /L(n)], \rho_{nL(n)})$ number of independent
random walks. We would like to show the probability that the
maximum of the end positions of these random walks is bigger than
$(\alpha-1-\vep)\sqrt{n}\log n /(4\beta)$ is asymptotically~1. By
Lemma \ref{lemma:comparison}, this probability is not increased if
we assume that all these random walks are started at 0 or 1,
depending on whether $[n^\alpha]-nL(n)$ is even or odd. But since
the random walks have $nL(n)$ steps to go, by relation
\eqref{eq:conv_outer}, no matter whether the starting point is 0
or 1, for large $n$, the probability that each random walk is to
the right of $(\alpha-1-\vep)/(4\beta)\cdot \sqrt{n}\log n $ is
asymptotically $n^{-(\alpha-1-\vep)}$. However we have at least
$X\sim\text{Bin}([n^\alpha /L(n)], \rho_{nL(n)})$ number of
\hbox{i.i.d.} trials, and by relation(\ref{eq:survivalToGenM}) and
Chernoff bound (\cite{Chernoff52} or \cite{AV79}), for all $n$
sufficiently large,
\begin{equation}\label{eq:Chernoff}
  P\left(X\leq \frac{1}{2}\cdot \frac{n^{\alpha}}{L(n)}\frac{2}{nL(n)\sigma^2}\right)
  \leq \exp\left(- \frac{n^{\alpha}}{L(n)}\frac{2}{nL(n)\sigma^2}\cdot \frac{1}{9}\right)
  \to 0.
\end{equation}
It follows that the probability for the maximum of the end
positions of these random walks to be bigger than
$(\alpha-1-\vep)/(4\beta)\cdot \sqrt{n}\log n $ is asymptotically
1.
\end{proof}

Proposition \ref{prop:lower_bd} gives the desired lower bound. We
now prove the upper bound.
\begin{prop}\label{prop:upper_bd}For any $\vep>0$,
$$
   \lim_{n\to\infty} P\left(\left.R^{(n)}_{[n^\alpha]}\leq \frac{\alpha-1+\vep}{4\beta}\cdot\sqrt{n}\log n
   \,\right|G^{(n)}_{[n^\alpha]}\right)=1.
$$
\end{prop}
\begin{proof}
For any $\vep_n\to 0$, define
$H^{(n)}_{[n^\alpha]}=\{Z^{(n)}_{[n^\alpha]-n}\geq
([n^\alpha]-n)\cdot \vep_n \}$.  Applying Lemmas
\ref{lemma:conditioning} and \ref{lemma:timeJiggle} once we see
that we can change the conditioning event from
$G^{(n)}_{[n^\alpha]}$ to $H^{(n)}_{[n^\alpha]}$; applying these
lemmas again  we see that we can change the conditioning event to
$G^{(n)}_{[n^\alpha]-n}$. Since $\alpha>1$, by relation
\eqref{eq:conv_outer}, the probability that each random walk is to
the right of $(\alpha-1+\vep/2) /(4\beta)\cdot\sqrt{n}\log n$ at
time $[n^\alpha]-n$ is asymptotically $n^{-(\alpha-1+\vep/2)}$.
Thus, using relations \eqref{eq:fundamental} and
\eqref{eq:exp_surv}, the conditional expectation of the number of
particles to the right of $(\alpha-1+\vep/2)\sqrt{n}\log n
/(4\beta)$  in generation $[n^\alpha]-n$ is
$$\aligned
   &E\left(Z^{(n)}_{[n^\alpha]-n}\, |\, G^{(n)}_{[n^\alpha]-n}\right)\cdot
   P\left(S_{[n^\alpha] -n} \geq (\alpha-1+\vep/2)
   /(4\beta)\cdot\sqrt{n}\log n\right)\\
  \sim&\frac{\sigma^2([n^\alpha] -n) }{2}\cdot n^{-(\alpha-1+\vep/2)}
  \sim  \frac{n^{1-\vep/2}\ \sigma^2}{2}.
\endaligned
$$
However, by relation (\ref{eq:survivalToGenM}), the probability
that a Galton-Watson process survives to time $n$ is  $\sim
2/(n\sigma^2)$, hence the number of particles to the \emph{right}
of $(\alpha-1+\vep/2)/(4\beta)\cdot \sqrt{n}\log n $ in generation
$[n^\alpha]-n$  whose families survive to time $[n^\alpha]$  has
expectation asymptotically equivalent to $ n^{-\vep/2}$, which
goes to 0. Therefore if we denote by
$$
    \aligned &R'^{(n)}_{[n^\alpha]}\\
    =&\text{the rightmost location in generation } [n^\alpha] \mbox{ of the descendants of the
         particles }\\
    &\text{which are to the  \emph{left} of } (\alpha-1+\vep/2)/(4\beta)\cdot\sqrt{n}\log n
     \text{ in generation } [n^\alpha]-n,
\endaligned$$
then it suffices to show further that
\begin{equation}\label{eq:R'_bd}
   P\left(\left.R'^{(n)}_{[n^\alpha]}\geq (\alpha-1+\vep)/(4\beta)\cdot\sqrt{n}\log n \,
   \right|\, G^{(n)}_{[n^\alpha]-n}\right)
    \to 0.
\end{equation}
By Lemma \ref{lemma:comparison}, this probability is not decreased
if we assume  all the particles  to the left of
$(\alpha-1+\vep/2)/(4\beta)\cdot \sqrt{n}\log n $  at time
$[n^\alpha]-n$ are located at $M_n$, where
\[\aligned
    M_n:=
    \left\{\aligned &\mbox{the biggest even number } \leq
                              (\alpha-1+\vep/2)/(4\beta)\cdot\sqrt{n}\log n , \\
                              &\hskip8.5cm\text{if }
                                [n^\alpha]-n \mbox{ is even};\\
                       &\mbox{the biggest odd number }\leq
                       (\alpha-1+\vep/2)/(4\beta)\cdot\sqrt{n}\log n , \\
                       &\hskip8.5cm\text{ if }
                       [n^\alpha]-n \mbox{ is odd}. \endaligned
                       \right.
\endaligned
\]
 In either
case, in order that $R'^{(n)}_{[n^\alpha]}\geq
(\alpha-1+\vep)/(4\beta)\cdot\sqrt{n}\log n $, since the ancestors
are to the left of $(\alpha-1+\vep/2)/(4\beta)\cdot\sqrt{n}\log n
$, at least one descendent will have to travel to the right at
least $\vep/(8\beta)\cdot\sqrt{n}\log n $ distance. Hence, since
the BRW is critical, we get
\begin{equation}\label{eqn:R'}
\aligned
   &P\left(R'^{(n)}_{[n^\alpha]}\geq (\alpha-1+\vep)/(4\beta)\cdot\sqrt{n}\log n\
    |\ G^{(n)}_{[n^\alpha]-n}\right) \\
   \leq &\ E(Z^{(n)}_{[n^\alpha]-n}\, |\, G^{(n)}_{[n^\alpha]-n})\cdot P^{M_n}
   (S_n\geq M_n + \vep/(8\beta)\cdot \sqrt{n}\log n
    ).
\endaligned
\end{equation}
 By Lemma
\ref{lemma:coupling},
\begin{equation}\label{eq:ldp}
   P^{M_n}\left(S_n\geq M_n + \vep/(8\beta)\cdot \sqrt{n}\log n \right)
   \leq P^{M_n}\left(\wt{S}_n\geq M_n +
     \vep/(8\beta)\cdot \sqrt{n}\log n\right).
\end{equation}
When $n$ is sufficiently large, $M_n$ will be bigger than
$\vep/(8\beta)\cdot \sqrt{n}\log n$, so by
Lemma~\ref{lemma:sy_srw} we get that the probability on the right
side of \eqref{eq:ldp} is bounded by
$C\exp\left(-b\vep^2/(64\beta^2)\cdot (\log n)^2\right)$. Using
\eqref{eqn:R'}, noting that $E(Z^{(n)}_{[n^\alpha]-n} |
G^{(n)}_{[n^\alpha]-n})$ $=O(n^\alpha)$ only grows polynomially in
$n$, we get~\eqref{eq:R'_bd}.
\end{proof}

\section{Proof of Theorem \ref{thm:measure_limit_super}}

We start with a simple observation. The following lemma about the
probabilities of survival is a supplement to the convergence in
\eqref{eq:feller_conv}.
\begin{lemma}\label{lemma:surv_prob_conv} For the total mass processes
$(Z^{(n)}_{[n^\alpha t]})_{t\geq 0}$ and the Feller diffusion
$(Y_t)_{t\geq 0}$, the following convergence holds:
\begin{equation}\label{eq:surv_prob_conv}
  P\left(Z^{(n)}_{[n^\alpha t]}>\delta n^\alpha\right)
  \to P\left(Y_t>\delta\right),\quad\mbox{ for all }\delta \geq 0 \mbox{ and for all }
  t> 0.
\end{equation}
\end{lemma}
\begin{proof}
For any $t>0$, the convergence in \eqref{eq:surv_prob_conv} when
$\delta>0$ follows from the marginal convergence
$Z^{(n)}_{[n^\alpha t]}/n^\alpha\Longrightarrow Y_t$ and that
$P(Y_t=\delta)=0$ (for any fixed $t>0$, by
\eqref{eq:survivalToGenM} and \eqref{eqn:lim_exp} it is easy to
show that the marginal distribution of $Y_t$ can be described as a
Poisson sum of exponentials, see, e.g., page 136 in
\cite{PerkinsSFnotes}, hence is continuous on $(0,\infty)$; see
also page 441 in \cite{RevuzYor} for an explicit density formula).
It remains to show
\[
  P\left(Z^{(n)}_{[n^\alpha t]}>0\right) \to P(Y_t>0).
\]
In fact, by the independence between the BRWs engendered by
different initial particles,
\[
  P\left(Z^{(n)}_{[n^\alpha t]}=0\right)
  =(1 - \rho_{[n^\alpha t]})^{Z^{(n)}_{0}},
\]
where $\rho_m$, as defined in \eqref{eq:survivalToGenM}, is the
probability that a Galton-Watson process started by a single
particle survives to generation $m$. By \eqref{eq:survivalToGenM}
and \eqref{eq:total_mass_conv},
\[
   (1 - \rho_{[n^\alpha t]})^{Z^{(n)}_{0}}
   \sim\exp\left(-\frac{2}{n^\alpha t \sigma^2}\cdot Z^{(n)}_{0} \right)
   \to \exp\left(-\frac{2y}{t\sigma^2}\right).
\]
The right side equals $P(Y_t=0)$, see, e.g., equation (II.5.12) in
\cite{PerkinsSFnotes}.
\end{proof}

\begin{proof}[Proof of Theorem \ref{thm:measure_limit_super}]

\bigskip\noindent\textbf{A. Convergence of Marginal distributions.}
We will show that for any fixed $t> 0$,  on the Skorokhod space $
D([0,\infty);\zz{R})$,
\begin{equation}\label{eq:marginal}
   \left(\frac{X_{[n^\alpha t]}^{(n)}([0,\sqrt{n}a])}{n^\alpha}\right)_{a\geq 0}
   \Longrightarrow \left(X_t([0,a])=Y_t\cdot\pi([0,a])\right)_{a\geq
   0}.
\end{equation}
Let  $L(n):=[(\log n)^2]$, and write
\[
  \frac{X_{[n^\alpha t]}^{(n)}([0,\sqrt{n}a])}{n^\alpha}
  = \frac{Z^{(n)}_{[n^\alpha t]- nL(n)}}{n^\alpha}\cdot
  \frac{X_{[n^\alpha t]}^{(n)}([0,\sqrt{n}a])}{Z^{(n)}_{[n^\alpha t-
  nL(n)]}}\cdot \indic_{\{Z^{(n)}_{[n^\alpha t]- nL(n)}>0\}}.
\]
For any $a\geq 0$ and $\delta>0$, we will show the following law of
large numbers:
\begin{equation}\label{eq:lln}
 \left(\frac{X_{[n^\alpha t]}^{(n)}([0,\sqrt{n}a])}{Z^{(n)}_{[n^\alpha t]-
  nL(n)}} - \pi([0,a])\right)\cdot \indic_{\{Z^{(n)}_{[n^\alpha t]- nL(n)}>\delta n^\alpha\}}
  \to 0. 
\end{equation}
\noindent{\bf Claim}: If this holds, then we have the
finite-dimensional convergence below: for any $ k\in \zz{N}$ and
any $0\leq a_1\leq\ldots\leq a_k<\infty,$
\begin{equation}\label{eq:marginal_fd}
   \left(\frac{X_{[n^\alpha
   t]}^{(n)}([0,\sqrt{n}a_i])}{n^\alpha}\right)_{a_1,\ldots,a_k}
   \Longrightarrow
   \left(Y_t\cdot\pi([0,a_i])\right)_{a_1,\ldots,a_k}.
\end{equation}
Note that the LHS and RHS of \eqref{eq:marginal} are both
increasing processes and the RHS is continuous,  by Theorem
VI.3.37 in \cite{JS03}, the above finite-dimensional convergence
implies the convergence \eqref{eq:marginal} as processes on
$[0,\infty)$.

We now prove the claim, which is a direct consequence of Lemma
\ref{lemma:surv_prob_conv}, \eqref{eq:feller_conv},   Slutsky's
theorem and \eqref{eq:lln}.  We shall only prove the convergence
for any single $a\geq 0$; the joint convergence can be proved
similarly. Let $f:\zz{R}\to\zz{R}$ be any bounded Lipschitz
continuous function. We want to show that
\begin{equation}\label{eq:lln_2_conv}
 E f\left(\frac{X_{[n^\alpha
   t]}^{(n)}([0,\sqrt{n}a])}{n^\alpha}\right) \to Ef(Y_t\cdot\pi[0,a]).
\end{equation}
In fact, for any $\vep>0$, there exists $\delta>0$ such that
\[
   P(0<Y_t\leq\delta)\leq \vep.
\]
By Lemma \ref{lemma:surv_prob_conv}, for all $n$ sufficiently large,
\[
   P(0<Z^{(n)}_{[n^\alpha t]- nL(n)}\leq\delta n^\alpha)\leq 2\vep.
\]
Hence, denote by $M=\max_x |f(x)|$,
\[\aligned
   &\left|E f\left(X_{[n^\alpha t]}^{(n)}([0,\sqrt{n}a])/n^\alpha\right)
         - Ef(Y_t\cdot\pi[0,a])\right|\\
   \leq&\left|f(0)\cdot P\left(Z^{(n)}_{[n^\alpha t]- nL(n)}=0\right) - f(0)\cdot
   P(Y_t=0)\right|    +3M\vep\\
   +&\left|E \left(f\left(\frac{Z^{(n)}_{[n^\alpha t]- nL(n)}}{n^\alpha}\cdot
   \frac{X_{[n^\alpha t]}^{(n)}[0,\sqrt{n}a]}{Z^{(n)}_{[n^\alpha t-
  nL(n)]}}\right)\indic_{\{Z^{(n)}_{[n^\alpha t]- nL(n)}>\delta
  n^\alpha\}}\right)
         \right.\\
         &\quad\left.- E \left(f\left(\frac{Z^{(n)}_{[n^\alpha t]- nL(n)}}{n^\alpha}\cdot
         \pi[0,a]\right)\indic_{\{Z^{(n)}_{[n^\alpha t]- nL(n)}>\delta n^\alpha\}}\right)\right|\\
   +& \left|E \left(f\left(\frac{Z^{(n)}_{[n^\alpha t]- nL(n)}}{n^\alpha}\cdot
         \pi[0,a]\right)\indic_{\{Z^{(n)}_{[n^\alpha t]- nL(n)}>\delta
         n^\alpha\}}\right)\right.\\
         &\quad\left.- E\left(f(Y_t\cdot\pi[0,a])\cdot
         \indic_{\{Y_t>\delta\}}\right)\right|\\
   :=& I + 3M\vep + II + III.
\endaligned
\]
By Lemma \ref{lemma:surv_prob_conv}, $I\to 0$. By
\eqref{eq:feller_conv} and  Slutsky's theorem, $III\to 0.$
Finally, $II\to 0$ by the Lipschitz continuity of $f$,
\eqref{eq:feller_conv}, \eqref{eq:lln} and the dominated
convergence theorem.

We now prove the  law of large numbers \eqref{eq:lln}, by using a
mean-variance calculation. Let $\mathcal{F}^{(n)}_{[n^\alpha t]-
nL(n)}$ be the configuration of the BRW at time $[n^\alpha t]-
nL(n)$, $\mathcal{Z}^{(n)}_{[n^\alpha t]- nL(n)}$ be the set of
particles at time $[n^\alpha t]- nL(n)$, and for each particle
$u_i=u_i^{(n)}\in \mathcal{Z}^{(n)}_{[n^\alpha t]- nL(n)}$, let
$x_i=x_i^{(n)}$ be its location (at time $[n^\alpha t]- nL(n)$),
$U^{u_i}_k(x)$ be its number of descendants at site $x$ at time $k +
[n^\alpha t]- nL(n)$, and $Z_k^{u_i}$ be its total number of
descendants at time $k+[n^\alpha t]- nL(n)$.

We start with the mean calculation.
\[\aligned
  &E \left(\frac{X_{[n^\alpha t]}^{(n)}[0,\sqrt{n}a]}{Z^{(n)}_{[n^\alpha t]- nL(n)}}
  \cdot \indic_{\{Z^{(n)}_{[n^\alpha t]- nL(n)}>\delta n^\alpha\}}\right)\\
 =& E\left(\frac{E\left(X_{[n^\alpha t]}^{(n)}[0,\sqrt{n}a]\,|\,\mathcal{F}^{(n)}_{[n^\alpha t]- nL(n)}\right)}
 {Z^{(n)}_{[n^\alpha t]- nL(n)}}\cdot \indic_{\{Z^{(n)}_{[n^\alpha t]- nL(n)}>\delta n^\alpha\}}\right)\\
\endaligned
\]
By relation \eqref{eq:fundamental},
\[\aligned
   &E\left(X_{[n^\alpha t]}^{(n)}[0,\sqrt{n}a]\,|\,\mathcal{F}^{(n)}_{[n^\alpha t-
   nL(n)]}\right)\\
   =&\sum_{i=1}^{Z^{(n)}_{[n^\alpha t]- nL(n)}}
   P(S_{nL(n) }\in
   [0,\sqrt{n}a]\,|\,S_0=x_i).
\endaligned
\]

By Lemma  \ref{lemma:comparison}, if we let
\[\left\{
\aligned
  p^{0}_{nL(n)}&:=P(S_{nL(n) }\in
   [0,\sqrt{n}a]\,|\,S_0=0)\\
   p^{1}_{nL(n)}&:=P(S_{nL(n) }\in
   [0,\sqrt{n}a]\,|\,S_0=1),
\endaligned \right.
\]
then
\begin{equation}\label{eq:bd_pb01}
P(S_{nL(n) }\in
   [0,\sqrt{n}a]\,|\,S_0=x_i)\leq
\left\{ \aligned
   p^{0}_{nL(n)},\quad&\mbox{if } x_i
   \mbox{ is even}\\
   p^{1}_{nL(n)},\quad&\mbox{if } x_i
   \mbox{ is odd}.\\
\endaligned \right.
\end{equation}
Therefore, by Proposition \ref{prop:RW_conv_rate} and Lemma
\ref{lemma:surv_prob_conv},
\begin{equation}\label{eq:lln_upper_bd}
\aligned
   &E\left(\frac{E\left(X_{[n^\alpha t]}^{(n)}[0,\sqrt{n}a]\,|\,\mathcal{F}^{(n)}_{[n^\alpha t]- nL(n)}\right)}
 {Z^{(n)}_{[n^\alpha t]- nL(n)}}\cdot \indic_{\{Z^{(n)}_{[n^\alpha t]- nL(n)}>\delta n^\alpha\}}\right)\\
 \leq & E\left(p^{0}_{nL(n)}\vee p^{1}_{nL(n)}\cdot \indic_{\{Z^{(n)}_{[n^\alpha t]- nL(n)}>\delta n^\alpha\}}\right)
 \to \pi[0,a]\cdot P(Y_t>\delta).
\endaligned
\end{equation}

On the other hand, by relation \eqref{eq:fundamental} again, for
any $C>0$,
\[
  E\left(\frac{1}{n^\alpha} X_{[n^\alpha t]- nL(n)}^{(n)}([C\sqrt{n},\infty)) \right)
  =\frac{1}{n^\alpha} \sum_{i=1}^{Z_{0}^{(n)}}
  P\left(S_{[n^\alpha t]- nL(n)} \geq C\sqrt{n}\,|\,S_0=x_{0;i}^{(n)}\right),
\]
where $Z_{0}^{(n)}$ is the total number of particles at time $0$,
and  $x_{0;i}^{(n)}$ is the location of the $i$th initial
particle. By the tightness of $\{X_{0}^{(n)}(\sqrt{n}\cdot)
/n^\alpha\}$ \eqref{eq:ini_measure_tight} and
\eqref{eq:conv_normal_range} we see for any $\vep>0$, there exists
$C>0$ such that for all $n$ sufficiently large,
\begin{equation}\label{eq:mess_right_side}
  E\left(\frac{1}{n^\alpha}  X_{[n^\alpha t]- nL(n)}^{(n)}([C\sqrt{n},\infty)) \right)
 \leq \vep.
\end{equation}
Therefore by Markov's inequality,
\begin{equation}\label{eq:outer_mass_neg}
  P\left(\frac{1}{n^\alpha}X_{[n^\alpha t]- nL(n)}^{(n)}([C\sqrt{n},\infty)) \geq \sqrt{\vep}\right)
  \leq \sqrt{\vep}.
\end{equation}
Now by Lemma  \ref{lemma:comparison} again, for those particles
$u_i$ at time $[n^\alpha t]- nL(n)$ which are \emph{to the left}
of $C\sqrt{n}$, if we let
\[\aligned
       \left\{\aligned  M_{n;even} &=\mbox{the biggest even number } \leq
                              C\sqrt{n}      ;\\
                       M_{n;odd} &=\mbox{the biggest odd number }\leq
                       C\sqrt{n}, \endaligned
                       \right.
\endaligned
\]
and
\[\left\{
\aligned
  p^{even}_{nL(n)}&:=P(S_{nL(n) }\in
   [0,\sqrt{n}a]\,|\,S_0=M_{n;even})\\
   p^{odd}_{nL(n)}&:=P(S_{nL(n) }\in
   [0,\sqrt{n}a]\,|\,S_0=M_{n;odd}),
\endaligned \right.
\]
then
\begin{equation}\label{eq:bd_pbeo}
P(S_{nL(n) }\in
   [0,\sqrt{n}a]\,|\,S_0=x_i)\geq
\left\{ \aligned
   p^{even}_{nL(n)},\quad&\mbox{if } x_i
   \mbox{ is even}\\
   p^{odd}_{nL(n)},\quad&\mbox{if } x_i
   \mbox{ is odd}.\\
\endaligned \right.
\end{equation}
Hence, by Proposition \ref{prop:RW_conv_rate}, Lemma
\ref{lemma:surv_prob_conv} and \eqref{eq:outer_mass_neg},
\begin{equation}\label{eq:lln_lower_bd}
\aligned
   &\liminf_n E\left(\frac{E\left(X_{[n^\alpha t]}^{(n)}[0,\sqrt{n}a]\,|\,\mathcal{F}^{(n)}_{[n^\alpha t]- nL(n)}\right)}
 {Z^{(n)}_{[n^\alpha t]- nL(n)}}\cdot \indic_{\{Z^{(n)}_{[n^\alpha t]- nL(n)}>\delta n^\alpha\}}\right)\\
   \geq &\liminf_n E\left(\frac{E\left(X_{[n^\alpha t]}^{(n)}[0,\sqrt{n}a]\,|\,\mathcal{F}^{(n)}_{[n^\alpha t]- nL(n)}
   \right)}
 {Z^{(n)}_{[n^\alpha t]- nL(n)}}\right.\\
   &\qquad \left.\cdot \indic_{\{Z^{(n)}_{[n^\alpha t]- nL(n)}>\delta n^\alpha;\ X_{[n^\alpha t]- nL(n)}^{(n)}
   ([C\sqrt{n},\infty))\leq n^{\alpha}\sqrt{\vep}\}}\right)\\
 \geq & \liminf_n E\left(\frac{Z^{(n)}_{[n^\alpha t]- nL(n)}-n^\alpha \sqrt{\vep}}{Z^{(n)}_{[n^\alpha t]- nL(n)}}
 \ p^{even}_{nL(n)}\wedge p^{odd}_{nL(n)}\right.\\
   &\qquad \left.\cdot \indic_{\{Z^{(n)}_{[n^\alpha t]- nL(n)}>\delta n^\alpha;\ X_{[n^\alpha t]- nL(n)}^{(n)}
   ([C\sqrt{n},\infty))\leq   n^{\alpha}\sqrt{\vep}\}}\right)\\
\geq &\left(1-\frac{\sqrt{\vep}}{\delta}\right)\cdot \pi[0,a]\cdot
\left(P(Y_t>\delta)-\sqrt{\vep}\right).
\endaligned
\end{equation}
By the arbitrariness of $\vep$, we get the desired lower bound
\[\aligned
  &\liminf_n E\left(\frac{E\left(X_{[n^\alpha t]}^{(n)}([0,\sqrt{n}a])\,|\,
  \mathcal{F}^{(n)}_{[n^\alpha t]- nL(n)}\right)}
 {Z^{(n)}_{[n^\alpha t]- nL(n)}}\cdot \indic_{\{Z^{(n)}_{[n^\alpha t]- nL(n)}>\delta n^\alpha\}}\right)\\
 \geq &\pi[0,a]\cdot
P(Y_t>\delta).
\endaligned
\]
So, combining it with \eqref{eq:lln_upper_bd}, we get the
convergence of expectation
\[
\lim_n E \left(\left(\frac{X_{[n^\alpha
t]}^{(n)}[0,\sqrt{n}a]}{Z^{(n)}_{[n^\alpha t]- nL(n)}}-\pi[0,a]\right) \cdot \indic_{\{Z^{(n)}_{[n^\alpha t]- nL(n)}>\delta n^\alpha\}}\right)\\
 = 0.
\]

It remains to show that
\begin{equation}\label{eq:var_2_0}
 \lim_n \var\left(\left(\frac{X_{[n^\alpha t]}^{(n)}[0,\sqrt{n}a]}{Z^{(n)}_{[n^\alpha t]- nL(n)}}
  - \pi[0,a]\right)\cdot \indic_{\{Z^{(n)}_{[n^\alpha t]- nL(n)}>\delta n^\alpha\}}\right)= 0.
\end{equation}
By conditioning on $\mathcal{F}^{(n)}_{[n^\alpha t]- nL(n)}$, we get
\[\aligned
   &\var\left(\left(\frac{X_{[n^\alpha t]}^{(n)}[0,\sqrt{n}a]}{Z^{(n)}_{[n^\alpha t]-
  nL(n)}} - \pi[0,a]\right)\cdot \indic_{\{Z^{(n)}_{[n^\alpha t]- nL(n)}>\delta n^\alpha\}}\right)\\
 =&E\left(\left(\frac{\var\left(X_{[n^\alpha t]}^{(n)}[0,\sqrt{n}a]\,|\, \mathcal{F}^{(n)}_{[n^\alpha t]- nL(n)}\right)}
 {\left(Z^{(n)}_{[n^\alpha t]-  nL(n)}\right)^2} \right)\cdot \indic_{\{Z^{(n)}_{[n^\alpha t]- nL(n)}>\delta n^\alpha\}}\right)\\
&\quad +
  \var\left(\left(\frac{E\left(X_{[n^\alpha t]}^{(n)}[0,\sqrt{n}a]\,|\, \mathcal{F}^{(n)}_{[n^\alpha t]- nL(n)}\right)}
  {Z^{(n)}_{[n^\alpha t]-  nL(n)}}  - \pi[0,a] \right)\cdot
  \indic_{\{Z^{(n)}_{[n^\alpha t]- nL(n)}>\delta n^\alpha\}}\right)\\
 :=& I + II.
\endaligned
\]
We will show that both terms converge to 0.

We start with term I. Recall that for each particle
$u_i\in\mathcal{Z}^{(n)}_{[n^\alpha t]- nL(n)}$, $U^{u_i}_k(x)$
denotes its number of descendants at site $x$ at time $k + [n^\alpha
t]- nL(n)$, and $Z_k^{u_i}$ is its total number of descendants at
time $k+[n^\alpha t]- nL(n)$.  By the independence between the BRWs
$U^{u_i}$ and \eqref{eq:var_Z_m},
\[\aligned
   &\var\left(X_{[n^\alpha t]}^{(n)}[0,\sqrt{n}a]\,|\, \mathcal{F}^{(n)}_{[n^\alpha t]- nL(n)}\right)\\
   =&\sum_{u_i\in \mathcal{Z}^{(n)}_{[n^\alpha t]- nL(n)}} \var\left(\sum_{x\in [0,\sqrt{n}a]}U^{u_i}_{nL(n) }(x) \right)\\
   \leq & \sum_{u_i\in \mathcal{Z}^{(n)}_{[n^\alpha t]- nL(n)}}
   E\left(Z_{nL(n) }^{u_i}\right)^2\\
   =&Z^{(n)}_{[n^\alpha t]- nL(n)}(1+ nL(n) \sigma^2).
\endaligned
\]
Hence
\[
  I\leq E\left(\frac{1+ nL(n) \sigma^2}{Z^{(n)}_{[n^\alpha t]-  nL(n)}}
  \cdot \indic_{\{Z^{(n)}_{[n^\alpha t]- nL(n)}>\delta n^\alpha\}}\right) \to 0.
\]

As to term II, by \eqref{eq:bd_pb01},
\[
   E\left(X_{[n^\alpha t]}^{(n)}([0,\sqrt{n}a])\,|\, \mathcal{F}^{(n)}_{[n^\alpha t]- nL(n)}\right)
   \leq Z^{(n)}_{[n^\alpha t]- nL(n)}\cdot p^{0}_{nL(n)}\vee p^{1}_{nL(n)}
\]
furthermore, on the event $\{X_{[n^\alpha t]-
nL(n)}^{(n)}([C\sqrt{n},\infty)) \leq n^{\alpha}\sqrt{\vep}\}$, by
\eqref{eq:bd_pbeo},
\[
   E\left(X_{[n^\alpha t]}^{(n)}([0,\sqrt{n}a])\,|\, \mathcal{F}^{(n)}_{[n^\alpha t]- nL(n)}\right)
   \geq \left(Z^{(n)}_{[n^\alpha t]-  nL(n)}- n^{\alpha}\sqrt{\vep}\right)\cdot
   p^{even}_{nL(n)}\wedge p^{odd}_{nL(n)}
\]
Hence,
\[\aligned
   &II\\
   \leq& E\left(\left(\frac{E(X_{[n^\alpha t]}^{(n)}([0,\sqrt{n}a])\,|\, \mathcal{F}^{(n)}_{[n^\alpha t]- nL(n)})}
  {Z^{(n)}_{[n^\alpha t]-  nL(n)}}  - \pi[0,a] \right)^2\cdot \indic_{\{Z^{(n)}_{[n^\alpha t]- nL(n)}>\delta n^\alpha\}}
  \right)\\
  \leq& \sqrt{\vep}\\
  +&  E\left(\max\left(\left(p^{0}_{nL(n)}\vee p^{1}_{nL(n)}  - \pi[0,a]\right)^2,
   \left((1-\sqrt{\vep}/\delta)\ p^{even}_{nL(n)}\wedge p^{odd}_{nL(n)} - \pi[0,a]\right)^2  \right) \right.\\
   &\quad\left.\cdot \indic_{\{Z^{(n)}_{[n^\alpha t]- nL(n)}>\delta n^\alpha;\ X_{[n^\alpha t]-  nL(n)}^{(n)}([C\sqrt{n},\infty)) \leq
  n^{\alpha}\sqrt{\vep}\}}\right)\\
  =&O(\sqrt{\vep}),
\endaligned
\]
where the term $\sqrt{\vep}$ in the second inequality comes from
\eqref{eq:outer_mass_neg}, and in the last equation we used
Proposition~\ref{prop:RW_conv_rate}. By the arbitrariness of
$\vep$, $II\to 0$ and hence \eqref{eq:var_2_0} holds.

\bigskip\noindent
{\bf B.  Convergence of Finite Dimensional Distributions.} This
follows from the Markov property and similar calculations as in
Part A.
\end{proof}

\section*{Acknowledgments}
The author thanks Steven P. Lalley for  helpful discussions.


\begin{thebibliography}{15}
\expandafter\ifx\csname
natexlab\endcsname\relax\def\natexlab#1{#1}\fi
\expandafter\ifx\csname url\endcsname\relax
  \def\url#1{\texttt{#1}}\fi
\expandafter\ifx\csname
urlprefix\endcsname\relax\def\urlprefix{URL }\fi

\bibitem[{Angluin and Valiant(1979)}]{AV79}
Angluin, D., Valiant, L.~G., 1979. Fast probabilistic algorithms
for
  {H}amiltonian circuits and matchings. J. Comput. System Sci. 18~(2),
  155--193.
\newline\urlprefix\url{http://dx.doi.org/10.1016/0022-0000(79)90045-X}

\bibitem[{Athreya and Ney(1972)}]{athreya72}
Athreya, K.~B., Ney, P.~E., 1972. Branching processes.
Springer-Verlag, New
  York, die Grundlehren der mathematischen Wissenschaften, Band 196.

\bibitem[{Chernoff(1952)}]{Chernoff52}
Chernoff, H., 1952. A measure of asymptotic efficiency for tests
of a
  hypothesis based on the sum of observations. Ann. Math. Statistics 23,
  493--507.

\bibitem[{Durrett et~al.(1991)Durrett, Kesten, and Waymire}]{DKW91}
Durrett, R., Kesten, H., Waymire, E., 1991. On weighted heights of
random
  trees. J. Theoret. Probab. 4~(1), 223--237.

\bibitem[{Etheridge(2000)}]{etheridge}
Etheridge, A.~M., 2000. An introduction to superprocesses. Vol.~20
of
  University Lecture Series. American Mathematical Society, Providence, RI.

\bibitem[{Feller(1939)}]{Feller39}
Feller, W., 1939. Die {G}rundlagen der {V}olterraschen {T}heorie
des {K}ampfes
  ums {D}asein in wahrscheinlichkeitstheoretischer {B}ehandlung. Acta Bioth.
  Ser. A. 5, 11--40.

\bibitem[{Feller(1951)}]{Feller51}
Feller, W., 1951. Diffusion processes in genetics. In: Proceedings
of the
  {S}econd {B}erkeley {S}ymposium on {M}athematical {S}tatistics and
  {P}robability, 1950. University of California Press, Berkeley and Los
  Angeles, pp. 227--246.

\bibitem[{Jacod and Shiryaev(2003)}]{JS03}
Jacod, J., Shiryaev, A.~N., 2003. Limit theorems for stochastic
processes, 2nd
  Edition. Vol. 288 of Grundlehren der Mathematischen Wissenschaften
  [Fundamental Principles of Mathematical Sciences]. Springer-Verlag, Berlin.

\bibitem[{Kac(1947)}]{Kac47}
Kac, M., 1947. Random walk and the theory of {B}rownian motion.
Amer. Math.
  Monthly 54, 369--391.

\bibitem[{Kesten(1995)}]{Kesten95}
Kesten, H., 1995. Branching random walk with a critical branching
part. J.
  Theoret. Probab. 8~(4), 921--962.

\bibitem[{Lalley and Zheng(2007)}]{lz07}
Lalley, S.~P., Zheng, X., 2007. Occupation statistics of critical
branching
  random walks. {\rm to appear in } Ann. Probab.
\newline\urlprefix\url{arXiv:0707.3829v3}

\bibitem[{Lawler and Limic(2007)}]{lawler07}
Lawler, G.~F., Limic, V., 2007. Random Walk: A Modern
Introduction.
  http://www.math.uchicago.edu/~lawler/srwbook.pdf.

\bibitem[{Perkins(2002)}]{PerkinsSFnotes}
Perkins, E., 2002. Dawson-{W}atanabe superprocesses and
measure-valued
  diffusions. In: Lectures on probability theory and statistics
  ({S}aint-{F}lour, 1999). Vol. 1781 of Lecture Notes in Math. Springer,
  Berlin, pp. 125--324.

\bibitem[{Revuz and Yor(1999)}]{RevuzYor}
Revuz, D., Yor, M., 1999. Continuous martingales and {B}rownian
motion, 3rd
  Edition. Vol. 293 of Grundlehren der Mathematischen Wissenschaften
  [Fundamental Principles of Mathematical Sciences]. Springer-Verlag, Berlin.

\bibitem[{Watanabe(1968)}]{watanabe68}
Watanabe, S., 1968. A limit theorem of branching processes and
continuous state
  branching processes. J. Math. Kyoto Univ. 8, 141--167.

\end{thebibliography}
\end{document}